\documentclass[12pt]{amsart}
\usepackage{amsmath}
\usepackage{amsfonts}
\usepackage{amssymb}
\usepackage{amsthm}
\usepackage[all]{xy}
\usepackage{color}
\usepackage{verbatim}
\usepackage{graphicx}
\usepackage{tikz}
\usepackage{placeins}
\usepackage{float}
\usepackage{listings}
\usepackage{tikz}
\usetikzlibrary{matrix}
\usetikzlibrary{positioning}
\usepackage{empheq}
\usepackage{caption}
\usepackage{cases}
\usepackage{epsfig}

\newtheorem{theorem}{Theorem}[section]
\newtheorem{lemma}[theorem]{Lemma}

\newtheorem{corollary}[theorem]{Corollary}
\newtheorem{definition}[theorem]{Definition}

\numberwithin{equation}{section}

\author[lokenath Kundu, Kaustav Mukherjee]{Lokenath Kundu, Kaustav Mukherjee}

\email{bholaktw2010@gmail.com, lokenath$\_$kundu@srmap.edu.in}

\address{SRM University, A.P.}
\address{Indian Institute of Science Education and Research Bhopal, Madhya Pradesh 462066 }

\title[Symmetry of surfaces for linear fractional group] {Symmetry of surfaces for linear fractional group}

\date{24/11/21}
\begin{document}
	
	\begin{abstract}
		We will compute the stable upper genus for the family of finite non-abelian simple groups $PSL_2(\mathbb{F}_p)$ for $p \equiv 3~(mod~4)$. This classification is well-grounded in the other branches of Mathematics like topology, smooth, and conformal geometry, algebraic categories.
		\end{abstract}
	\maketitle	
	\section{Introduction}
\noindent	Let $\Sigma_g$ be a Riemann surface of genus $g\geq 0$. We will imply by the action of a finite group $G$ on $\Sigma_g$, a properly discontinuous, orientation preserving, faithful action. The collection $\lbrace g \geq 0| G ~\text{acts on}~ \Sigma_g \rbrace$ is known as spectrum of $G$ denoted by $Sp(G)$. The least element of $Sp(G)$ is denoted by $\mu(G)$ familiar as the minimum genus of the group $G$. An element $g \in Sp(G)$ is said to be the stable upper genus of a given group $G$, if $g+i \in Sp(G)$ for all $i \in \mathbb{N}$. The necessary and sufficient condition for an effective action of a group $G$ preserving the orientation on compact, connected,  orientable surface $\Sigma_g$ of genus $g$ except for finitely many exceptional values of $g$ was proved by Kulkarni in \cite{kulkarni}. In particular the group $PSL_2(\mathbb{F}_p)$  has the above mentioned property for $p \geq ~ 5$, and $p$ is odd. The authors determined the minimum genus for the family of finite groups in \cite{ming2,ming1}. \\
\noindent Any action of a finite group $G$ on a Riemann surface $\Sigma_g$ of genus $g$ gives an orbit space $\Sigma_h ~ := \Sigma_g/G$ also known as orbifold. We can take this action as conformal action, that means the action is analytic in some complex structure on $\Sigma_g$, as the positive solution of Nielson Realization problem \cite{niel,eck} implies that if any group $G$ acts topologically on $\Sigma_g$ then it can also act conformally with respect to some complex structure. \\
\noindent The orbit space $\Sigma_h$ is again a Riemann surface possibly with some marked points and the quotient map $p~:~\Sigma_g~\rightarrow~\Sigma_h$ is a branched covering map. Let $B=~\lbrace c_1,c_2,\dots,c_r~ \rbrace$ be the set of all branch points in $\Sigma_h$ and $A:=p^{-1}(B)$. Then  $p:~\Sigma_g \setminus A ~\rightarrow ~\Sigma_h \setminus B$ is a proper covering. The tuple $(h;m_1,m_2,\dots,m_r)$ is known as signature of the finite group $G$, where $m_1,m_2,\dots,m_r$ are the order of stabilizer of the preimages of the branch points  $c_1,c_2,\dots,c_r$ respectively. By Riemann-Hurwitz formula we have $$ (g-1)=~|G|(h-1)+\frac{|G|}{2}\sum_{i=1}^r(1-\frac{1}{m_i}) \label{R.H.formula}.$$ The signature of a group encodes the information of the group action of a Riemann surface and about $Sp(G)$. For more details about signature of Fuchsian group and Riemann surfaces refer to \cite{otto}, and \cite{sve} respectively. In \cite{kundu1,kundu2}, with accurate use of Frobenius theorem and explicit formation of surface kernel epimorphisms, the author able to prove the following theorems:

	\begin{theorem}\label{1}\cite{kundu1}
		$ ( h;2^{[a_{2}]}, 3^{[a_{3}]}, 4^{[a_{4}]}, 7^{[a_{7}]} ) $ is a signature of $ PSL_2(\mathbb{F}_7) $ if and only if $$  1+168(h-1)+ 42a_{2} + 56a_{3} + 63a_{4} + 72a_{7} \geq 3 $$ except when the signature is $(1;2)$.
	\end{theorem}

\begin{theorem}\label{2}\cite{kundu1}
	$ ( h;2^{[a_{2}]}, 3^{[a_{3}]}, 5^{[a_{5}]}, 6^{[a_6]} 11^{[a_{11}]} ) $ is a signature of $ PSL_2(\mathbb{F}_{11}) $ if and only if $$  1+660(h-1)+ 165a_{2} + 220a_{3} + 264a_{5} + 275a_6 +300a_{11} \geq 26 .$$
\end{theorem}
and the following lemma;
\begin{lemma}\label{3}\cite{kundu2}
	$(h_{\geq ~ 0};~ 2^{[a_2]},~ 3^{[a_3]},~ 4^{[a_4]},~ 5^{[a_5]},~ d^{[a_d]},~ \frac{p-1}{2}^{[a_{\frac{p-1}{2}}]},~ \frac{p+1}{2}^{[a_{\frac{p+1}{2}}]},~ p^{[a_p]})$ is a signature for $PSL_2(\mathbb{F}_p)$ for $p ~ \equiv ~ 3 ~ (mod ~ 4)$ if and only if $$2(h-1)+~\frac{a_2-1}{2}~ + \frac{2a_3-1}{3} + ~ \frac{3a_4}{4} +~ \frac{4a_5}{5} +~ \frac{(d-1)a_d+1}{d} ~+ \frac{a_{\frac{p-1}{2}}(p-3)}{p-1} ~+ \frac{a_{\frac{p+1}{2}}(p-1)}{p+1} $$ $$+\frac{(p-1)a_p}{p} ~ \geq 0 \text{ or }$$  $$20(h-1) ~ + 10[\frac{a_2}{2} ~ +\frac{2.a_3}{3} ~+\frac{3.a_4}{4} ~+\frac{4.a_5}{5} ~+\frac{(d-1)a_d}{d} ~+\frac{(p-3)a_{\frac{p-1}{2}}}{p-1} ~+$$ $$\frac{(p-1)a_{\frac{p+1}{2}}}{p+1} ~+\frac{(p-1)a_p}{p} ] ~ \geq ~ 1 $$ when $p ~ \geq ~ 13, ~ p \equiv  \pm 1~(\mod ~ 5~),~ p ~ \not \equiv ~ \pm ~ 1(\mod ~ 8), ~ \text{and} ~ d \geq 15$. Here $$d:=min\lbrace e|e\geq 7 \text{ and either } e|\frac{p-1}{2} \text{ or } e|\frac{p+1}{2} \rbrace.$$
\end{lemma}

\noindent Having the details knowledge of the spectrum of the group $PSL_2(\mathbb{F}_p)$ one would like to address the following question:\\
\noindent \textbf{What is the stable upper genus for each of the group $PSL_2(\mathbb{F}_p)$ for $p\equiv 3~(mod ~4)$?} In \cite{kundu1}, we find out the stable upper genus for the group $PSL_2(\mathbb{F}_7)$ is 399 and the stable upper genus for the group $PSL_2(\mathbb{F}_{11})$ is 3508  using generic programming techniques \cite{ipython,pandas,matplotlib,numpy}. Following a similar approach described in \cite{kundu1}, here we will largely extend the scenario for higher prime numbers and determine the stable upper genus value for the each of the members of the family of finite groups $PSL_2(\mathbb{F}_p)$ for $p \equiv 3~(mod~4)$. Interestingly, the novelty of this work is the observance of the exponential curve fitting for the stable upper genus values of $PSL_2(\mathbb{F}_p)$ for $p\equiv 3~(mod~4)$ which has not been seen in earlier cases \cite{kulkarni,kundu1}.  \\
\noindent Here we have stated the main result of this paper as follows:\\
\noindent \begin{theorem} \label{main}
	The stable upper genus value of the group $PSL_2(\mathbb{F}_p)$ can be written in the form
	\begin{equation}
		g=a p^b e^{c\times p},
		\label{g_exp}
	\end{equation}
	where $a$, $b$ and $c$ are constants discussed in the proof and $g$ represents the upper stable genus of the group $PSL_2(\mathbb{F}_p)$ while $p$ is the respective prime for $p \equiv 3 ~(mod ~4)$.
\end{theorem}

\noindent Implementing computations with loops over large variations of $h$ and $a_i$ [\ref{1},\ref{2},\ref{3}] by means of Python coding \cite{ipython,pandas,numpy}, we find a set of stable upper genus values of $PSL_2(\mathbb{F}_p)$ for $p\in\{7,11,19,23\}$ which we discuss in the following sections. Based on the set of stable upper genus values, we construct a mathematical function described in Eq. \ref{g_exp}, which follows the variation in the stable upper genus values of $PSL_2(\mathbb{F}_p)$ with the respect to $p$. We discuss the detailed comparison of the expression in Eq. \ref{g_exp} with the dependency of the stable upper genus on $p$ in the proof. To explore the possibility of obtaining a mathematical function describing the stable upper genus as a function of $p$ for the group $PSL_2(\mathbb{F}_p)$, we make use of the curve-fitting technique on Mathematica \cite{mathematica} following from Fit and Manipulate tool, which provides us with the best fit on the data set of the stable upper genus corresponding to respective prime $p\in\{7,11,19,23\}$. We have specifically considered the function type for the stable upper genus as 
\begin{equation}
	g=a p^b \exp[cp],
\end{equation}
where $a$, $b$ and $c$ are constants that are obtained based on the best fit on the data-set and $p$ is the prime following $p\equiv 3~(mod~4)$. This expression subsequently provides us an estimate along with upper bound of stable upper genus of the group $PSL_2(\mathbb{F}_p)$ for general $p\equiv 3~(mod~4)$.
\noindent We have organized our paper in the following way. In chapter 2 we will study the necessary preliminary results. In most cases, we will state the theorems without proof. In chapter 3, we will prove our main Theorem [\ref{main}].  

\section{preliminaries}
\noindent In this section, we will collect the knowledge about the properly discontinuous actions of a group $G$ on any Riemann surface $\Sigma_g$, signature of a finite group, the family of groups $PSL_2(\mathbb{F}_p)$ for a prime $p$, curve fitting, exponential fitting.

\noindent We start with the definition of properly discontinuous action of a finite group on a Riemann surface.
\begin{definition}\cite{sve}
	Let $G$ be a finite group is said to act on a Riemann surface $\Sigma_g$ properly discontinuously if for any $x\in  \Sigma_g$ there exists a neighbouhood $U$ of $x$ in $X$ such that $g(U)\cap U=\emptyset$ for only finitely many $g\in G$.  
\end{definition}
\subsection{Fuchsian group}
A discrete subgroup of the Fuchsian group is known as Fuchsian group \cite{sve}.
\begin{theorem}\cite{sve}
	A group $\Gamma$ is a Fuchsian group if and only if $\Gamma$ acts on the upper half plane $\mathbb{H}$ properly discontinuously.
\end{theorem}
\begin{definition}
	A Fuchsian group $\Gamma$ is said to be co-compact Fuchsian group if $\mathbb{H}/\Gamma$ is compact.
\end{definition} 
\subsection{Dirichlet Region}
Let $\Gamma$ be a Fuchsian group acts on the upper half plane $\mathbb{H}$. Let $p \in \mathbb{H}$ be a point which is not fixed by any non identity element of $\Gamma \setminus \lbrace id \rbrace.$ The Dirichlet region center at $p$ for $\Gamma$ is defined as $$D_p(\Gamma)=\lbrace z\in \mathbb{H}|\rho(z,p)\leq \rho(z,T(p)) ~ \forall T\in \Gamma \setminus \lbrace id \rbrace \rbrace$$
\noindent Here $\rho$ is the usual hyperbolic metric. \begin{theorem}
	The Dirichlet region $D_p(\Gamma) $is a connected region of $\Gamma$ if $p$ is not fixed by any element of $\Gamma \setminus \lbrace id \rbrace . $
\end{theorem}
\begin{proof}
	\cite{sve}.
\end{proof}
\begin{theorem}
	Any two distinct points that lie inside the Dirichlet region will belong to two different $\Gamma$ orbits.
\end{theorem}
\begin{proof}
	\cite{sve}.
\end{proof}
\noindent Two points $w_1,w_2\in \mathbb{H}$ are said to be congruent if they lie to the same $\Gamma$ orbit. Any two pints in a fundamental region $F$ may be congruent only if the points lie in the boundary of $F$. Let $F$ be a Dirichlet region for a Fuchsian group $\Gamma$. We will consider all congruent vertices of $F$. The congruence is an equivalence relation on the vertices of $F$, the equivalence classes are called the \textbf{cycles}. Let $w\in \mathbb{H}$ be fixed by an elliptic element $T$ of $\Gamma$, then $Sw$ is fixed by $STS^{-1}$. So if one vertex of the cycle is fixed by an elliptic element then all the vertices of the cycle are fixed by the conjugate of the elliptic cycles. Those cycles are called elliptic cycles, and the vertices of the cycles are known as elliptic vertics. The cardinality of the collection of distinct elliptical cycles is same as the of non-congruent elliptic points in the Dirichlet region $F$. \\
\noindent Every non trivial stabilizer of any point in $\mathbb{H}$ is a maximal finite cyclic subgroup of the group $\Gamma$. In this context we have the following theorem.   
\begin{theorem}
	Let $\Gamma$ be a Fuchsian group, and $F$ be a Dirichlet region for $\Gamma$. Let $\alpha_1,\alpha_2, \dots, \alpha_n$ be the internal angles at all congruent vertices of $F$. Let $k$ be the order of the stabilizer in $\Gamma$ of one of the vertices. Then $\alpha_1+\alpha_2+\dots+\alpha_n=\frac{2\pi}{k}$. 
\end{theorem}
\begin{proof}
	\cite{sve}.
\end{proof}
\begin{definition}
	The orders of non-conjugate maximal finite cyclic subgroups of the Fuchsian group $\Gamma$ are known as the period of $\Gamma$.
\end{definition}

\subsection{Signature of Fuchsian group}
Let a Fuchsian group $\Gamma$ acts on $\mathbb{H}$. Let the area of the orbit space $\mathbb{H}/\Gamma$ has the finite area $i.e.~\mu(\mathbb{H}/\Gamma)<\infty .$ The restriction of the natural projevtion map $\mathbb{H}\rightarrow \mathbb{H}/\Gamma$ to the Dirichlet region $F$, identifies the congruent points of $F$. So $F/ \Gamma$ is an oriented surface possibly with some marked points  as the congruent points are lying on the boundary of $F$. The marked points are correspond to the elliptic cycles and the cusps are corresponding to the non-congruent vertices at infinity. As a space $\mathbb{H}/\Gamma$ is known as orbifold. The number of cusps and the genus of the orbifold decisive the topology type of the orbifold. The area of $\mathbb{H}/\Gamma$ is defined as the area of the fundamental region $F$. If one Dirichlet region is compact then all the other Dirichlet regions are compact. If a Fuchsin group has a compact Dirichlet region then the Dirichlet region has finitely many sides and the orbifold is also compact. \\
\noindent If a convex fundamental region for a Fuchsian group $\Gamma$ has finitely many sides then the Fuchsian group is known as geometrically finite group. 
\begin{theorem}
	Let $\Gamma$ be a Fuchsian group. If the orbifold $\mathbb{H}/\Gamma$ has finite area then the $\Gamma$ is geometrically finite.
\end{theorem}
\begin{proof}
	\cite{sve}.
\end{proof}
\begin{definition}{\textbf{(Co-compact Fuchsian group)}}
	A Fuchsian group is said to be co-compact if the orbifold $\mathbb{H}/\Gamma$ is compact topological space.
\end{definition}  
\noindent Let $\Gamma$ be a Fuchsian group and $F$ be a compact Dirichlet region for $\Gamma$. So the number of sides, vertices, and elliptic cycles of $F$ are finitely many. Let $m_1,m_2,\dots,m_r$ be the finite number of periods of $\Gamma$. Hence the orbifold $\mathbb{H}/\Gamma$ is a compact oriented surface of genus $g$ with $r$-many marked points. The tuple $(g;m_1,m_2,\dots,m_r)$ is known as the signature of the Fuchsian group $\Gamma$. 
\subsection{Signature of finite group}
Now we define the signature of a finite group in the sense of Harvey \cite{har}.

\begin{lemma}[Harvey condition]
	\label{Harvey condition}
	A finite group $G$ acts faithfully on $\Sigma_g$ with signature $\sigma:=(h;m_1,\dots,m_r)$ if and only if it satisfies the following two conditions: 
	
	\begin{enumerate}
		
		\item The \emph{Riemann-Hurwitz formula for orbit space} i.e. $$\displaystyle \frac{2g-2}{|G|}=2h-2+\sum_{i=1}^{r}\left(1-\frac{1}{m_i}\right), \text{ and }$$
		
		\item  There exists a surjective homomorphism $\phi_G:\Gamma(\sigma) \to G$ that preserves the orders of all torsion elements of $\Gamma$. The map $\phi_G$ is also known as surface-kernel epimorphism.
	\end{enumerate}
\end{lemma}
\begin{corollary}
	Let $Sig(G)$ denote the set of all possible signatures of a finite group $G$, then $Sig(G)$ and $Sp(G)$ have bijective correspondence via the Harvey condition.   
\end{corollary}
\subsection{The family of finite groups $PSL_2(\mathbb{F}_p)$}
Let $p$ be a prime number. The set $$PSL_2(\mathbb{F}_p):=\large\lbrace \begin{pmatrix}
	a & b \\
	c & d 
\end{pmatrix}|~ad-bc=1,~a,b,c,d \in \mathbb{F}_p \large\rbrace/ \pm I$$ forms a group under matrix multiplication. It is a simple linear group generated by two elements, $A=\begin{pmatrix}
	0 & 1 \\
	-1 & 0 
\end{pmatrix}$ of order $2$, and $B=\begin{pmatrix}
	0 & 1 \\
	-1 & -1 
\end{pmatrix}$ of order $3.$ The order of $AB= \begin{pmatrix}
	1 & 1 \\
	0 & 1 
\end{pmatrix}$ is $7, i.e.$ $$PSL_2(\mathbb{F}_p)=\langle A,B|A^2=B^3=(AB)^P \rangle.$$
\begin{theorem}
	Let $p$ be an odd prime. Let $G:=\langle x,y|x^p=y^p=(x^ay^b)^2=1,ab \equiv 1(mod~p) \rangle$
	be a two generator group. Then $G$ is isomorphic $PSL_2(\mathbb{F}_p).$
\end{theorem}
\begin{proof}
	\cite{beetham}.
\end{proof}
\subsubsection{Maximal subgroups of $PSL_2(\mathbb{F}_p)$}

The group $PSL_2(\mathbb{F}_p)$ has $\frac{p(p^2-1)}{2}$ many elements. The elements of the group $PSL_2(\mathbb{F}_p)$ have one of the following order  $p,~2,~3,~4,~\text{or}~5,~d $ and a divisor of either $\frac{p-1}{2}$ or $\frac{p+1}{2}$ where $d$ is defined as $$d= min \lbrace ~ e| ~ e \geq 7 \text{ and either } e| \frac{p-1}{2} \text{ or } ~ e| \frac{p+1}{2}  \rbrace.$$

\noindent A subgroup $H$ of $G$ is said to be a maximal subgroup of $G$ if there exists a subgroup $K$ such that $H \subset K \subset G,$ then either $H=K$ or $K=G.$    
The maximal proper subgroups of $PSL_2(\mathbb{F}_p)$ are the followings \cite{sjerve}; 
\begin{itemize}
	\item[1.] dihedral group of order $p-1$ or $p+1$.
	\item[2.] solvable group of order $\frac{p.(p-1)}{2}$.
	\item[3.] $A_4$ if $p \equiv 3,13,27,37 ~ (mod ~ 40)$.
	\item[4.] $S_4$ if $p \equiv \pm 1 ~ (mod ~ 8)$.
	\item[5.] $A_5$ if $p \equiv \pm 1 ~ (mod ~ 5)$.
\end{itemize}
\subsection{Exponential Regression}
\begin{definition}
	Exponential regression is defined as the process of obtaining a mathematical expression for the exponential curve that best fits a set of data. In \cite{exponentialregression}, an exponential regression model has been discussed. As an example, we know a data is fit into a linear regression, if it can be explained using $y=mx+c$ where the data is represented as $\{x,y\}$ with $m$ as the slope and $c$ is the intercept on $y$-axis. Similarly, if the set of data can be best explained using
	\begin{eqnarray}
		Log[y]&=mLog[x]+c\\
		Y&=mX+c
	\end{eqnarray}
	where $Y=Log[y]$ and $X=Log[x]$ with slope $m$ and intercept $c$ then it can be called as exponential regression. The above example is the simplest form of exponential regression, with possibilities of significant extension in more complex scenario.
\end{definition}
\section{Stable upper genus of $PSL_2(\mathbb{F}_p)$ for $p\equiv 3~(mod~4)$}
\noindent In this section we will prove our main theorem [\ref{main}] using python coding.
\begin{theorem}\label{19}
	The stable upper genus of the group $PSL_2(\mathbb{F}_{19})$ is 33112. 
\end{theorem}
\begin{proof}
	We will prove the theorem in two steps.
	\begin{enumerate}
		\item[Step 1:] We will first prove that $33111 \notin Sp(PSL_2(\mathbb{F}_{19})).$ \\
		\noindent From [\ref{3}] we know that $(h;2^{[a_2]},3^{[a_3]},5^{[a_5]},9^{[a_9]},10^{[a_{10}]},19^{[a_{19}]})$ is a signature of $PSL_2(\mathbb{F}_{19})$ if and only if $$3420h-3419+855a_2+1140a_3+1368a_5+1520a_9+1539a_{10}+1620a_{19}\geq 96.$$
	\noindent 	If possible let $$33111=3420h-3419+855a_2+1140a_3+1368a_5+1520a_9+1539a_{10}+1620a_{19}.$$
\noindent Then the value of $h$ could be at most $11$. Similarly the values of $a_i$ could be at most $43,~ 33,~ 27,~ 25,~24,~23$ for $i= ~ 2,~ 3,~ 5,~ 9,~10,~19$ respectively. So We will consider  $$0 ~ \leq ~ h ~ \leq ~11$$ $$0 ~ \leq ~ a_2 ~ \leq ~ 43$$ $$0 ~ \leq ~ a_3 ~ \leq ~ 33$$ $$0 ~ \leq ~ a_5 ~ \leq ~ 27$$ $$0 ~ \leq ~ a_9 ~ \leq ~ 25$$ $$0 ~ \leq ~ a_{10} ~ \leq ~ 24$$ $$0 ~ \leq ~ a_{19} ~ \leq ~ 23.$$

\noindent We execute the following python code to conclude that $PSL_2(\mathbb{F}_{19})$ can not act on a compact, connected, orientable surface of genus $33111$ preserving the orientation.

\lstset{language=Python}
\lstset{frame=lines}
\lstset{caption={$33111$ is not an admissable signature of $PSL_2(\mathbb{F}_{19})$}}
\lstset{label={2nd:code_direct}}
\lstset{basicstyle=\footnotesize}
\begin{lstlisting}
	def func2(h,a2,a3,a5,a9,a10,a19):
	return 1+3420*(h-1) + 855*a2 + 1140*a3 + 1368*a5 + 1520*a9 +
	
	 1539*a10 + 1620*a19
	
	
	
	for h in range(11):
	for a2 in range(43):
	for a3 in range(33):
	for a5 in range(27):
	for a9 in range(25):
	for a10 in range(24):
	for a19 in range(23):
	sol = func2(h,a2,a3,a5,a9,a10,a19)
	if sol >33111:
	if sol < 33111:
	if sol == 33111:
	print("wrong")
	\end{lstlisting}
\item[Step 2:] To complete the proof of our claim, we have to find out signatures corresponding to the genus values $33112-33967$ of $PSL_2(\mathbb{F}_{19})$. We execute the following python code to compute all the signature values of $PSL_2(\mathbb{F}_{19})$ corresponding to the genus values $33112-33967$. 

\lstset{language=Python}
\lstset{frame=lines}
\lstset{caption={Signatures of $PSL_2(\mathbb{F}_{19})$} corresponding to the genus value $33112-33967$}
\lstset{label={3rd:code_direct}}
\lstset{basicstyle=\footnotesize}
\begin{lstlisting}
	def func2(h,a2,a3,a5,a9,a10,a19):
	return 1+3420*(h-1) + 855*a2 + 1140*a3 + 1368*a5 + 1520*a9 +
		
	 1539*a10 + 1620*a19
	
	
	sol_arr = []
	const_arr = []
	for h in range(11):
	for a2 in range(44):
	for a3 in range(33):
	for a5 in range(27):
	for a9 in range(25):
	for a10 in range(25):
	for a19 in range(24):
	sol = func2(h,a2,a3,a5,a6,a11)
	if sol >33112:
	if sol < 33967:
	#print(sol)
	sol_arr += [sol]
	const_arr += [[h,a2,a3,a5,a9,a10,a19]]
	
	
	
	color_dictionary = dict(zip(sol_arr, const_arr))
	
	sort_orders = sorted(color_dictionary.items(), key=lambda x: x[0])
	
	for i in sort_orders:
	print(i[0], i[1])
	
\end{lstlisting} 

\noindent Now we have to prove that $PSL_2(\mathbb{F}_{19})$ can act on all compact, connected, orientable surface of genus $g ~ \geq ~ 33967$ preserving the orientation. Let $g ~ \geq 33967$, and $\Sigma_{g}$ be a compact, connected, orientable surface of genus $g$. So we have $$  g-33112 ~ \equiv ~ s ~ (mod ~855) ~ \text{ where } ~1 ~ \leq ~ s ~ \leq 854.$$ Then $g ~ = ~ l+n.855$ where $ l ~= 33112+ s$. We know the signature corresponding to the genus $l$ as $333112~\leq l~ \leq 33967$ and let it be $(h;m_2,~m_3,~m_5,~m_9,m_{10},m_{19})$. Then the signature corresponding to the genus $g$ is $(h;m_2+n,~m_3,~m_5,~m_9,m_{10},m_{19})$. In this way we can find signature corresponding to genus $g ~ \geq 33967$. This completes the proof of our claim. 
	\end{enumerate}
\end{proof}
\begin{theorem}\label{23}
	The stable upper genus of the group $PSL_2(\mathbb{F}_{23})$ is 297084. 
\end{theorem}
\begin{proof}
	Similar to Theorem\ref{19}.
\end{proof}
\begin{theorem}\label{31}
	The stable upper genus of the group $PSL_2(\mathbb{F}_{31})$ is 20275804. 
\end{theorem}
\begin{proof}
	Similar to Theorem\ref{19}.
\end{proof}
\textbf{Proof of the main theorem \ref{main}.}
\begin{proof}
	In the previous Theorems~\ref{19},\ref{23},\ref{31}, we have obtained the upper stable genus for $p\in\{19,23,31\}$. Combining with the previously obtained results described in \cite{kundu1}, we now posses a data-set of stable upper genus $g\in\{399,3508,33112,297084,20275804\}$ corresponding to prime $p\in\{7,11,19,23,31\}$ for the group $PSL_2(\mathbb{F}_p)$, shown as blue dots in Fig.\ref{fitting}. We next investigate the dependencies of the stable upper genus values $g$ with respect to prime number $p\equiv 3~(mod~4)$, subsequently constructing a mathematical function described in Eq.\ref{g_exp}, which can be visualized in Fig.\ref{fitting} \cite{mathematica}.
	\begin{figure}[htb]
		\centering
		\epsfig{file=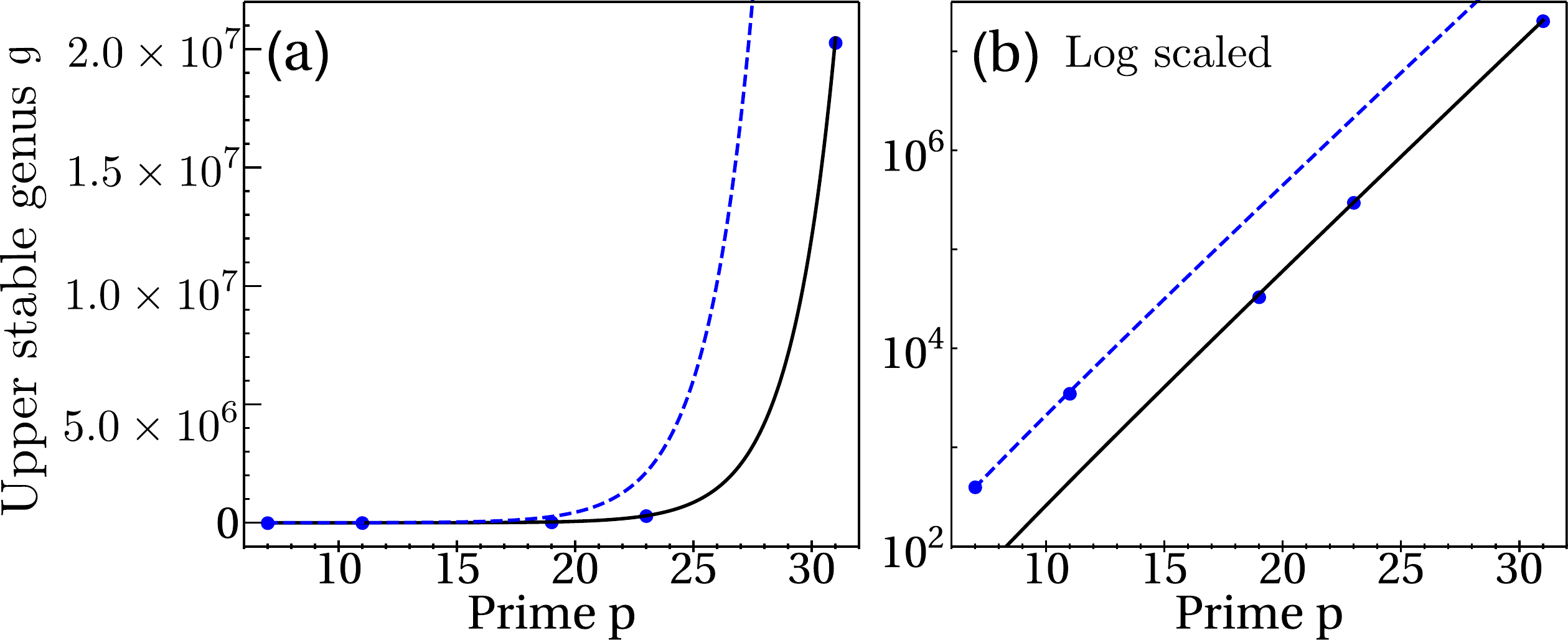,width= 0.9\linewidth}
		\caption{ Upper stable genus values (blue dot) given by $PSL_2(\mathbb{F}_p)$ for $p$ ranging from $7$ to $31$. Black line indicates the fitting on the upper stable genus values.
			\label{fitting}}
	\end{figure}
	In Fig.~\ref{fitting}~(b), we show the Log scaled plot of (a), which provides an indication of exponential dependence of $g$ on $p$. For the next step, we consider an ansatz in the form Eq. \ref{g_exp}. Additionally, Fig.~\ref{fitting} indicates two linear Log scaled lines, separating the cases $p\in\{7,11\}$ and $p\in\{19,23,31\}$. Based on the exponential fitting, using the following Mathematica code
	\lstset{language=Mathematica}
	\lstset{frame=lines}
	\lstset{caption={Exponential fitting of $g$ with respect to $p$ for the group $PSL_2(\mathbb{F}_{p})$}}
	\lstset{label={3rd:code_direct}}
	\lstset{basicstyle=\footnotesize}
	\begin{lstlisting}
		data = {{7, 399}, {11, 3508}, {19, 33112}, {23, 297084}, {31, 
				20275804}}
		
		lineardata = ListPlot[data, Frame -> True, 
		PlotMarkers -> {\[FilledCircle], 10},
		PlotStyle -> Blue, 
		FrameLabel -> {Style["Prime p", Black, 14], 
			Style["Upper stable genus g", Black, 14]}, 
		FrameTicksStyle -> Directive[FontSize -> 14], 
		ImageSize -> Medium, 
		PlotRange -> All]
		
		logdata = ListPlot[data, Frame -> True, 
		PlotMarkers -> {\[FilledCircle], 10},
		PlotStyle -> Blue, 
		FrameLabel -> {Style["Prime p", Black, 14], 
			Style["Upper stable genus g", Black, 14]}, 
		FrameTicksStyle -> Directive[FontSize -> 14], 
		ImageSize -> Medium, 
		PlotRange -> All]
		
		Manipulate[
		Show[{logdata, LogPlot[a*p^b*Exp[c*x], {p, 7, 31}]}], 
		{b, 0.0, 2, 0.5}, {a, 0.5, 4.5, 0.1}, {c, 0.5, 1, 0.01}]
		
		Linearplot = 
		Show[{lineardata, 
			Plot[0.5*x^0.5*Exp[0.51*x], {x, 7, 31}, 
			PlotRange -> All, 
			PlotStyle -> Black], 
			Plot[4.5*x^0.5*Exp[0.5*x], {x, 7, 31}, 
			PlotRange -> All, 
			PlotStyle -> {Blue, Dashed}]}, ImageSize -> Large]
		
		Logplot = 
		Show[{logdata, 
			LogPlot[0.5*x^0.5*Exp[0.51*x], {x, 7, 31}, 
			PlotRange -> All, 
			PlotStyle -> Black], 
			LogPlot[4.5*x^0.5*Exp[0.5*x], {x, 7, 31},
			PlotRange -> All, 
			PlotStyle -> {Blue, Dashed}]}, ImageSize -> Large]
		
	\end{lstlisting} 
	which leads us to the values of constants $\{a\rightarrow 4.5,b\rightarrow0.5,c\rightarrow0.5\}$ for $p\in\{7,11\}$ and $\{a\rightarrow 0.5,b\rightarrow0.5,c\rightarrow0.51\}$ for $p\in\{19,23,31\}$. This provides a general expression for $g$ as function of $p$ for higher range of primes $p\equiv 3~(mod~4)$ following the second category, given as
	\begin{equation}
		g=0.5p^{0.5}\exp[0.51p].
		\label{g_2}
	\end{equation}
	This expressions essentially captures the variation of upper stable genus $g$ with respect to $p$, and provides us with a rough estimate for general case scenario of $p\equiv 3~(mod~4)$ as shown in the Table\ref{table1}. In order to validate the fitting, we predict for arbitrary prime $p=59$, which should have stable upper genus close to $g=44907302712962$.
	\begin{table}[htb]
		\begin{center}
			\begin{tabular}{||c| c| c||} 
				\hline
				Prime p & Stable upper genus g & Exponential fitting g \\ [0.5ex] 
				\hline\hline
				7 & 399 & 394 \\ 
				\hline
				11 & 3508 & 3651  \\
				\hline
				19 & 33112 & 35209  \\
				\hline
				23 & 297084 & 297926 \\
				\hline
				31 & 20275804 & 20457219  \\ [1ex] 
				\hline
			\end{tabular}
			\caption{Comparison of Stable upper genus obtained computationally with the exponential fitting described in \ref{g_exp} for primes $p\in\{7,11,19,23,31\}$.}
			\label{table1}
		\end{center}
	\end{table}
\end{proof}
\section{Acknowledgement}
We acknowledge fruitful discussions with Dr. Manish Kumar Pandey. Author Mukherjee acknowledges financial support from Ministry of Education, India for the Prime Minister's Research Fellowship (PMRF) for pursuing Ph.D. Author Kundu is grateful to Council of Scientific and Industrial Research (CSIR),  India for the partial financial support. 
 
	\end{document}